\documentclass[a4paper,12pt]{article}
\usepackage{amsmath,amsthm,amssymb,latexsym,epsfig,graphicx}
\usepackage{enumerate}

\title{{\bf Beltrami equation with coefficient in  Sobolev and Besov spaces }}
\author{\Large{\Large Victor Cruz, Joan Mateu and  Joan Orobitg }}
\newtheorem{teorema}{Theorem}
\newtheorem*{teor}{Theorem}
\newtheorem{prop}{Proposition}

\newtheorem{co}{Corollary}
\newtheorem{lemma}[co]{Lemma}

\theoremstyle{definition}

\newtheorem*{gracies}{Acknowledgements}

\newcommand{\Rn}{{\mathbb R}^n}
\newcommand{\ep}{\varepsilon}
\newcommand{\C}{\mathbb{C}}

\newcommand{\fr}{\partial\,\Omega}

\begin{document}

\date{}

\maketitle

\begin{abstract}

Our goal in this work is to present some function spaces on the complex plane $\C$, $X(\C)$, for which the quasiregular solutions of the Beltrami equation, $\overline\partial f (z) = \mu(z) \partial f (z)$, have first derivatives locally in $X(\C)$, provided that the Beltrami coefficient $\mu$ belongs to $X(\C)$.

\end{abstract}

\section{Introduction}

A function $f : \C \longrightarrow \C$ is called $\mu$-quasiregular if it  belongs to the Sobolev space  $W^{1,2}_{\text{loc}}(\C)$ (functions with  distributional first order derivatives locally in $L^2$)  and satisfies the Beltrami equation
 \begin{equation}\label{eq1}
\overline\partial f (z) = \mu(z) \partial f (z), \quad a.e. \; z \in \C\,,
\end{equation}
where $\mu$, called the Beltrami coefficient of $f$, is a Lebesgue measurable function on the complex plane $\C$ satisfying $\|\mu\|_\infty < 1$\,.  
If, in addition, $f$ is a homeomorphism, then we say that f is $\mu$-quasiconformal.
Quasiconformal and quasiregular mappings are a central tool in modern geometric function theory and have had strong impact in other areas. 

It is well-known that quasiregular functions are locally in some H\"older class (Mori's Theorem), and moreover they actually belong to $W^{1,p}_{\text{loc}}$ for some $p>2$. 
In this paper we are interested in studying how the regularity of the Beltrami coefficient affects the regularity of the solutions of  \eqref{eq1}. Thus, if the Beltrami coefficient $\mu$ belongs to the 
H\"older class $C^{l,s}$, $0<s<1$, using Schauder estimates  (see for instance \cite[chapter 15]{AIM}),   then  $\mu$-quasiregular functions belong to $C^{l+1,s}_{\text{loc}}$. For the borderline cases $s=0$ and $s=1$, the $C^{l+1,s}$ regularity fails (e.g. \cite[p. 390]{AIM}). 
If $\mu\in W^{1,p}$, $2<p <\infty$, then one can read in Ahlfors' book  \cite[p. 56]{Ah} the result that quasiregular functions are locally in $W^{2,p}$.
The cases  $\mu\in W^{1,p}$, $p\le 2$, were studied in \cite{CFMOZ}; for instance, when $p=2$ one gets that the solutions are locally in  $ W^{2,q}$ for every $q< 2$.

Our goal in this work is to present some function spaces  $X$ for which all quasiregular solutions of \eqref{eq1} have first derivatives locally in $X$, provided that the Beltrami coefficient belongs to $X$.
These function spaces will enjoy the additional property of being an algebra (that is, the product of two functions in $X$ is again in $X$) and this feature will play an important role in our arguments. 
We deal with Triebel-Lizorkin spaces $F^{s}_{p,q}(\C)$ and Besov spaces  $B^{s}_{p,q}(\C)$ with $s>0$, $1<p<\infty$, $1<q<\infty$ and $sp>2$. 
Let $A^{s}_{p,q}(\C)$ denote any of these function spaces with the indices as we have determined.
In any case, the condition $sp>2$ ensures that we have  bounded continuous functions and multiplication algebras (e.g. \cite[4.6.4]{RS}).
In  Section \ref{preli} we will give the precise definitions of these function spaces  involved in the statement of the our first theorem.

\begin{teorema} \label{T1}
Suppose that $\mu \in A^{s}_{p,q}(\C)$ is compactly supported with $\| \mu \|_{\infty}= k<1$. Then any $f\in W^{1,2}_{\text{loc}}(\C)$ satisfying the Beltrami equation \eqref{eq1}
 has first derivatives locally in $A^{s}_{p,q}(\C)$.
\end{teorema}

When the Beltrami coefficient is compactly supported there is a unique $W^{1,2}_{\text{loc}}(\C)$
solution of \eqref{eq1} normalized by the condition $z + O (1/z)$ near $\infty$. Moreover, it is a homeomorphism of the complex plane. It is called the principal solution of \eqref{eq1}. By
Stoilow's Factorization Theorem (e.g. \cite[section 5.5]{AIM}), for any quasiregular function $f$ there exists a holomorphic function $h$ such that $f= h\circ \phi$, where $\phi$ is the associated principal solution. Therefore, we will only concentrate  on principal solutions.
As is well known, $\phi$ is given explicitly by the formula \cite[p. 165]{AIM}
$$
\phi(z) = z + \mathsf C(h)(z)\,,
$$
where the operator
\begin{equation}\label{cauchy}
 \mathsf C \, h(z) = \frac{1}{\pi}\,  \int_{\mathbb C} h(z-w)\frac{1}{w}\, \mathrm{d}w
\end{equation}
 is the Cauchy transform of $h$. When $h\in L^p$, $1<p<\infty$, one has the identity $\bar\partial \mathsf C(h) =h$. Consequently, our theorem immediately follows from next proposition.

\begin{prop} \label{Prop1}
Suppose that $\mu$ is compactly supported with $\| \mu \|_{\infty} = k<1$ and $\phi(z) =  z + \mathsf C(h)(z)$ is the principal solution of the Beltrami equation \eqref{eq1}. Let $s > 0$, $1 < p <\infty$, $1 < q <\infty$ and $sp > 2$. If  $\mu \in A^{s}_{p,q}(\C)$, then  $h \in A^{s}_{p,q}(\C)$.
\end{prop}

\textsc{Sketch of the proof.} The Beurling transform is the principal value convolution operator

$$
Bf(z)= - \frac{1}{\pi}\,   \text{p.v.}  \int_{\mathbb C} f(z-w)\frac{1}{w^2}\, \mathrm{d}w\,.
$$
The Fourier multiplier of $B$ is $\frac{\overline{\xi}}{\xi}$\,, or,
in other words,
$$
\widehat{Bf}(\xi) = \frac{\overline{\xi}}{\xi}\,\, \hat{f}(\xi)\,.
$$
Thus $B$ is an isometry on $L^2(\C)$\, and is well-known that $B$, as any Calder\'{o}n-Zygmund convolution operator, is bounded on $ A^{s}_{p,q}(\C)$. 

Recall the relation between the Cauchy and the Beurling
transforms: $\partial \mathsf C = B$. Thus, $\partial\phi = 1+ B(h)$ and $\overline{\partial}\,{\phi} =h$, and consequently the function $h$ is determined by the equation
$$
(I-\mu\,B)(h) = \mu\, .
$$
So, we only need to invert the Beltrami operator $I-\mu\,B$ on the corresponding function space.
This task is completed in Section  \ref{sec3}. 

\qed
\medskip

For the critical case $sp = 2$, we consider a Riesz potential space $I_1(L^{2,1}(\mathbb C))$, 
the set of functions with first order derivatives in the Lorentz space $L^{2,1}(\mathbb C)$. Even 
though close to $L^2$ , the Lorentz space $L^{2,1}(\mathbb C)$ is strictly contained in $L^2$. This small  improvement on the derivatives allows us to have continuous functions vanishing at 
infinity (by the way, remind that functions with first order derivatives in $L^{2}$ may 
not be continuous).
\begin{prop} \label{Prop2}
Suppose that $\mu\in I_1(L^{2,1}(\mathbb C))$ is compactly supported with $\| \mu \|_{\infty}= k<1$ and $\phi(z) =  z + \mathsf C(h)(z)$ is the principal solution of the Beltrami equation \eqref{eq1}. Then  $h \in I_1(L^{2,1}(\mathbb C))$.
\end{prop}
As we mentioned ago, Proposition \ref{Prop2}  does not hold when the Beltrami coefficient 
only has first derivatives in $L^{2}$ . However, the analogous result would remain valid 
if we replace  $I_1(L^{2,1}(\mathbb C))$ by $I_s(L^{\frac{2}{s},1}(\mathbb C))$, $0 < s < 2$.

The main result of \cite{MOV} identifies a class of non-smooth Beltrami coefficients which determine bilipschitz quasiconformal mappings. In particular, one proved the following result.

\begin{teor}[\cite{MOV}]
Let $\Omega$ be a bounded domain of $\C$ with boundary of class  $\mathcal{C}^{1,\varepsilon}$, $0< \varepsilon <1$, and let  $\mu\in \mathcal{C}^{0,\varepsilon}(\Omega)$ with $\| \mu\|_{\infty}<1$.
Let $\phi(z) =  z + \mathsf C(h)(z)$ be the principal solution of the Beltrami equation \eqref{eq1}.
Then  $h \in \mathcal{C}^{0,\varepsilon '}(\Omega)$ for any $\varepsilon' <\varepsilon$ and moreover $\phi$ is billipschitz .
\end{teor}

Now, we replace the H\"older smoothness of the Beltrami coefficient by a Sobolev (or Besov) condition restricted on a domain. (See definitions in the next section). 

\begin{teorema}\label{T2}
Let $0<s< \varepsilon<1$ and $1<p<\infty$ such that $sp>2$ and let $\Omega$ be a bounded domain of $\C$ with boundary of class  $\mathcal{C}^{1,\varepsilon}$.
Suppose that $\mu$ is  supported in $\overline\Omega$ with $\| \mu \|_{\infty}= k<1$ and $\phi(z) =  z + \mathsf C(h)(z)$ is the principal solution of the Beltrami equation \eqref{eq1}. \begin{enumerate}
\item If $\mu \in W^{s,p}(\Omega)$, then $h \in W^{s,p}(\Omega)$.
\item If  $\mu \in B^{s}_{p,p}(\Omega)$, then  $h \in B^{s}_{p,p}(\Omega)$.
\end{enumerate}
\end{teorema}

The proof runs in parallel to that of the above propositions, but now a new obstacle appears:
the boundedness of the Beurling transform on $W^{s,p}(\Omega)$ (or $B^{s}_{p,p}(\Omega)$).
In general, it is not clear if Calder\'{o}n-Zygmund convolution operators are bounded on $W^{s,p}(\Omega)$
(or $B^{s}_{p,p}(\Omega)$). Of course, the answer depends on the operator and on the boundary of the domain. We  will study this question in domains $\Omega $ of $\mathbb R^{n}$, $n\ge 2$.

In $\mathbb R^n$ we consider the kernel $K(x)= \frac{\omega(x)}{|x|^n}$, $x\ne 0$,  where $\omega$ is a homogeneous function of degree $0$, with zero integral on the unit sphere and $\omega \in \mathcal{C}^{1}(S^{n-1})$. Then, the singular integral
$$
T f(x)=   \text{p.v.}  \int f(y) K(x-y)\, \mathrm{d}y\,
$$
is bounded on $L^p(\mathbb R^n)$, $1<p<\infty$. (Really, the condition  $\omega \in \mathcal{C}^{1}(S^{n-1})$ could be weakened
but it is enough for our purpose). On the other hand, Sobolev spaces $W^{s,p}(\mathbb R^n)$ ($=F^{s}_{p,2}(\mathbb R^n)$) are described as spaces of Bessel potentials, that is, $f\in W^{s,p}$ if and only if $f=G_{s} * g$, where $G_{s}$ denotes the Bessel kernel of order $s$ and $g\in L^p$ (e.g. \cite[chapter 5]{St}).   Remember that the Bessel kernel of order $s$, $G_{s}$ , is the $L^1$ function with 
Fourier transform $(1+ |\xi|^2)^{-\frac{s}{2}}$. Then, because $T$ is a convolution operator, one has the identity
$$
T(f) = T(G_{s}*g) = G_{s}*(Tg)
$$
and one gets the boundedness of $T$ on $W^{s,p}$, $1<p< \infty$. But if one takes $f\in W^{s,p}(\Omega)$, $\Omega$ a domain of $\mathbb R^n$, then
$$
T_{\Omega} f(x) :=   \text{p.v.}  \int_{\Omega} f(y) K(x-y)\, dy\,
$$
clearly belongs to $L^p(\Omega)$. However, perhaps $T_{\Omega}f\notin W^{s,p}(\Omega)$. For instance, let $Q$ denote a rectangle in $\C$ and $\chi_Q$ denote its characteristic function.  A computation shows that the Beurling transform of $\chi_{Q}$, $B \chi_{Q}$, has logarithmic singularities at the vertices of the rectangle and, therefore, its first derivatives belong to $L^p(Q)$ only if $p<2$ (e.g. \cite[p. 147]{AIM}). For positive results, we restrict our attention to operators with even kernel, that is, $K(-x) = K(x)$. In Section \ref{CZdo} we will deal with Theorem \ref{T3}.

\begin{teorema}\label{T3}
Let $\Omega$ be a bounded domain of $\mathbb R^n$ with boundary of class  $\mathcal{C}^{1,\beta}$, $\beta >0$, and let $T$ be an even smooth homogeneous Calder\'{o}n-Zygmund operator.
\begin{enumerate}
\item If $T\chi_{\Omega} \in B^{s}_{p,p}(\Omega)$, $0<s<1$, $n<sp<\infty$, then $T_{\Omega}: B^{s}_{p,p}(\Omega) \longrightarrow B^{s}_{p,p}(\Omega)$.
\item If $T\chi_{\Omega} \in W^{s,p}(\Omega)$,  $0<s<1$, $n<sp<\infty$, then $T_{\Omega}: W^{s,p}(\Omega) \longrightarrow W^{s,p}(\Omega)$.
\item If $T\chi_{\Omega} \in W^{1,p}(\Omega)$, $n<p <\infty$, then $T_{\Omega}: W^{1,p}(\Omega) \longrightarrow W^{1,p}(\Omega)$.
\end{enumerate}

In any case the norm operator depends on the domain $\Omega$ and the Calder\'{o}n-Zygmund constant of the kernel of $T$ (see \eqref{CZ} for the definition). 
\end{teorema}

The result reduces the study of the boundedness of the operator $T_{\Omega}$ to the behaviour of $T_{\Omega}$ on the function $\chi_{\Omega}$. Thus, we have a necessary and sufficient condition of type $T(1)$.  In the proof of Theorem 3, we follow the same method of Y.\ Meyer in  \cite{Me}, where he studied the continuity of generalised Calder\'{o}n-Zygmund operators on Sobolev spaces $W^{s,p}(\mathbb R^n)$.

Since T is bounded on $L^p$, using complex and real interpolation, one could think that items 1 and 2 of the above theorem are a consequence of the third one. But this it not the case because the conditions on items 1 and 2 are weaker than $T\chi_{\Omega} \in W^{1,p}(\Omega)$. 
When $\Omega$ is a bounded domain of $\mathbb R^n$ with boundary of class $ \mathcal{C}^{1,\varepsilon}$, $0<s<\varepsilon <1$, and  $n<sp<\infty$ then $T_{\Omega}$ is bounded on $W^{s,p}(\Omega)$ and $B^{s}_{p,p}(\Omega)$ (see details in Section \ref{CZdo}). In particular, the assumptions on the domain $\Omega$, in the statement of Theorem \ref{T2},  are to ensure that the Beurling transform is bounded on the corresponding function space. Recently, V. Cruz and X. Tolsa( \cite{CT}, \cite{To}) have showed that if the outward unit normal $N$ on $\partial\Omega$ belongs to the Besov space $B^{s -1/p}_{p,p}(\partial\Omega)$, then $B {\chi_{\Omega}}\in W^{s,p}(\Omega)$.\\

In Section 2 we shall introduce some basic notation and set up some necessary 
preliminaries. The proof of Proposition \ref{Prop1} and  Proposition \ref{Prop2}  are in  Section \ref{sec3}. In Section \ref{CZdo} we study 
even smooth homogeneous Calder\'{o}n-Zygmund operators on domains. The proof of the Theorem \ref{T2} is explained in Section 5. 

As usual, the letter $C$ will denote a constant, which may be different at each occurrence and which is independent of the relevant variables under consideration.

\section{Preliminaries}\label{preli}
We start reviewing some basic facts concerning Triebel-Lizorkin spaces and Besov 
spaces. Let $\mathcal S(\mathbb R^n)$  be the usual Schwartz class of rapidly decreasing 
$\mathcal C^{\infty}$-functions  and 
$\widehat g$ stands for the Fourier transform of $g$. Let $\psi\in \mathcal S(\mathbb R^n)$ with 
$\widehat\psi(\xi) =1$ if $|\xi| \le 1$ and $\widehat\psi(\xi) =0$ if $|\xi | \ge 3/2$. We set $\psi_{0} = \psi$ and $\widehat\psi_{j}(\xi) = \widehat\psi(2^{-j}\xi) -  \widehat\psi(2^{-j+1}\xi)$, $j\in\mathbb N$.
Since $\sum_{j=0}^{\infty}\widehat\psi_{j}(\xi)=1$ for all $\xi\in \mathbb R^n$, the $\widehat\psi_{j}$
form a dyadic resolution of unity. Then, for $f\in L^1_{\operatorname{loc}}(\mathbb R^n)$, $1\le p,q <\infty$, and $s>0$, one defines the norms
$$
\| f\|_{B^s_{p,q}} = \left(  \sum_{j=0}^{\infty} \| 2^{js}\psi_{j}* f \|_{p}^q    \right)^{\frac{1}{q}}
$$
and
$$
\| f\|_{F^s_{p,q}} = \left\| \left(  \sum_{j=0}^{\infty} | 2^{js}\psi_{j}* f |^q     \right)^{\frac{1}{q}}
\right\|_{p}
$$
The Besov space  $B^s_{p,q}(\mathbb R^n)$ consists of the functions such that $\| f\|_{B^s_{p,q}} <\infty$, while the functions in the Triebel-Lizorkin space $F^s_{p,q}(\mathbb R^n)$ are those such that $\| f\|_{F^s_{p,q}} <\infty$.

The spaces $F^s_{p,2}$  , $1 < p <\infty$, are known as Sobolev spaces of fractional order or 
Bessel-potential spaces and we prefer denote them by $W^{s,p}$. Since $p\ge 1$ and $q \ge 1$, 
both $B^s_{p,q}$ and $F^s_{p,q}$ are Banach spaces. A systematic treatment of these spaces may  be found in \cite{Tri1}, \cite{RS} and \cite[Chapter 6]{Gr}. A remarkable fact when $sp > n$ is that $B^s_{p,q}$ and $F^s_{p,q}$ form an algebra with respect to pointwise multiplication, that is, 
\begin{equation}\label{algebra1}
\| f \cdot g\|_{A^s_{p,q}} \le C \| f\|_{A^s_{p,q}}  \| g\|_{A^s_{p,q}},
\end{equation}
where $A^s_{p,q}$ denotes the corresponding Besov space or Triebel-Lizorkin space (e.g. \cite[4.6.4]{RS}). Moreover, functions in these spaces satisfy some H\"older condition and so they are continuous functions with
\begin{equation*}\label{algebra2}
\| f\|_{\infty}\le C \| f\|_{A^s_{p,q}}\, .
\end{equation*}

We say that a bounded domain
$\Omega \subset \Rn$ has a boundary of class $\mathcal{C}^{1,\varepsilon}$ if
$\fr$ is a $C^1$ hyper-surface whose unit normal vector satisfies
a Lipschitz (H\"older) condition of order $\ep$ as a function on the
surface. To state an alternative condition, for $x = (x_1, \dots ,x_n)\in \Rn $ we use the notation $x = (x',x_n)$\,, where $x'= (x_1,\dots,x_{n-1})$\,. Then $\Omega$ has a boundary of class
$\mathcal{C}^{1,\varepsilon}$ if for each point $a \in \partial\,\Omega$ one may find a ball $B(a,r)$ and a function $x_n = \varphi(x')$, of class $\mathcal{C}^{1,\varepsilon}$, such that, after a rotation if necessary, $\Omega \cap B(a,r)$ is the part of $B(a,r)$ lying below the graph of $\varphi$\,. Thus we get
\begin{equation}\label{eq6}
\Omega \cap B(a,r)= \{x \in B(a,r) : x_n <
\varphi(x_1,\dots,x_{n-1})\}\,.
\end{equation}
We say that $\Omega$ is a bounded Lipschitz domain if the function $\varphi$ in \eqref{eq6}  is of class $\mathcal{C}^{0, 1}$. 

In general, if one has a function space $X$ defined on $\mathbb R^n$ and a domain $\Omega \subset \Rn$, one defines the space $X(\Omega)$ as the restrictions of functions of $X$ from $\mathbb R^n$ to $\Omega$. In addition, the restriction space is endowed with the quasi-norm 
quotient.  In the cases that we are considering we have an intrinsic characterization of elements of $X(\Omega)$. We will use these characterizations in the proofs of Theorems 2 and 3. 
Let $\Omega$ be a bounded Lipschitz domain in  $\mathbb R^n$, $1<p<\infty$ and $0<s<1$. Then:
\begin{enumerate}
\item $f\in B_{p,p}^{s}(\Omega)$ if and only if $f\in L^{p}(\Omega)$ and
 \[
\int_{\Omega}\int_{\Omega}\frac{|f(x)-f(y)|^{p}}{|x-y|^{n+s p}}\mathrm{d}x\mathrm{d}y<\infty.\]
(e.g. \cite[p. 169]{Tar})
\item $f\in W^{s,p}(\Omega)$ if and only if $f\in L^{p}(\Omega)$ and 
\begin{equation}\label{sobolevs}
\int_{\Omega}\left(\int_{\Omega}\frac{|f(x)-f(y)|^{2}}{|x-y|^{n+2 s}}\mathrm{d}x\right)^{\frac{p}{2}}\mathrm{d}y<\infty.\end{equation} (e.g \cite[p. 1051]{Str})
\item $f\in W^{1,p}(\Omega)$ if and only if $f\in L^{p}(\Omega)$ and 
\begin{equation}\label{sobolev1}
\lim_{\alpha \to0}\alpha\int_{\Omega}\int_{\Omega}\frac{|f(x)-f(y)|^{p}}{|x-y|^{n+p-\alpha}}\mathrm{d}x\mathrm{d}y<\infty.\end{equation} (e.g.  \cite[p. 703]{Br2})
\end{enumerate}

\medskip

A smooth (of class $\mathcal C^1$) homogeneous Calder\'{o}n-Zygmund operator is a principal value convolution operator of type 
\begin{equation*}\label{eq6bis}
T(f)(x)= \text{p.v.} \int f(y)\,K(x-y) \, \mathrm{d}y \,,
\end{equation*}
where
$$
K(x)= \frac{\omega(x)}{|x|^{n}}\,,\quad x \neq 0\,,
$$
$\omega(x)$ being a homogeneous function of degree~$0$,
continuously differentiable on $\Rn \setminus\{0\}$ and with zero
integral on the unit sphere. Note that one trivially has
$$
|K(x-y)|\leq \frac{C}{|x-y|^n}
$$  
and
$$
|K(x-y)-K(x-y')|\leq C\frac{|y-y'|}{|x-y|^{n+1}}\quad\text{whenever }|x-y|\geq2|y-y'| .
$$
The Calder\'{o}n-Zygmund constant of the kernel of $T$ is defined as
\begin{equation}
\| T \|_{CZ} = \|K(x)\,|x|^n \|_\infty + \|\nabla
K(x)\,|x|^{n+1}\|_\infty\,.\label{CZ}
\end{equation}
The operator $T$ is said to be even if the kernel is even, namely,
if $\omega(-x)=\omega(x)\,,$ \,for all $ x \neq 0\,.$ 
The even character of $T$ gives the cancellation $T(\chi_{B})\chi_{B}=0 $ for each ball $B$, 
which should be understood as a local version of the global cancellation property T (1) = 0 common to all smooth homogeneous  Calder\'on-Zygmund operators. This extra cancellation property is essential for proving Lemma \ref{lema:acotacion} and so Theorem \ref{T3}.

It is well known that  Calder\'on-Zygmund convolution operators are bounded on $L^p(\mathbb R^n)$ and also on $W^{s,p}(\mathbb R^n)$ (because $W^{s.p} = G_{s}*L^p$).
Using the method of real interpolation, one easily gets that these operators are also bounded on
 $B^s_{p,q}(\mathbb R^n)$ (see also  \cite[6.7.2]{Gr} for a direct proof).
The boundedness of   Calder\'on-Zygmund convolution operators on
 $F^s_{p,q}(\mathbb R^n)$ was proved in \cite[Theorem 3.7]{FTW} (see \cite[Theorem 1.2]{JHL} for a nice proof).  Summarizing, if $s>0$ and $1<p, q<\infty$ we have
 \begin{equation}\label{CZdes}
  \|T f\|_{A^s_{p,q}}\le C  \| f\|_{A^s_{p,q}},
\end{equation}
where $C$ is a constant which depends on $s,p,q,n$ and $\| T \|_{CZ}$. \\

Lorentz spaces are defined on measure spaces $(Y,m)$, but we only need the case $Y=\mathbb C$ and $m$ is the Lebesgue planar measure. The classical definition of Lorentz spaces use the rearrangement function. For any measurable function $f$ we define its nonincreasing rearrangement by
$$
f^{*}(t) := \inf \{ s:   m\{z\in\mathbb C\colon|f(z)|>s\} \le t   \}.
$$
For $1\leq p,q<\infty$, the Lorentz space $L^{p,q}(\mathbb C)$ is the set of functions $f$  such that $\|f\|_{L^{p,q}}<\infty $, with
\begin{displaymath}
\| f\|_{L^{p,q}(\mathbb C)}:= \left\{ \begin{array}{ll}
(\int_0^{\infty} [ t^{1/p} f^{*}(t)]^q t^{-1}dt)^{1/q}, & \textrm{for  $ 1\le q<\infty$} \\[1mm]
\sup_{t>0} t^{1/p} f^{*}(t), & \textrm{for $q=\infty$}  
\end{array}
\right .
\end{displaymath}
A second definition of Lorentz spaces, which is equivalent to the first one, is given by real interpolation between Lebesgue spaces: 
$$
(L^{p_0}, L^{p_1})_{\theta, q}=L^{p,q},
$$
where $1\le p_0 <p<p_1 \le\infty$, $1\le q\le\infty$, $0<\theta<1$ and $\frac{1}{p}=\frac{1-\theta}{p_0}+ \frac{\theta}{p_1}$. Lorentz spaces inherited from Lebesgue spaces the stability property of the multiplication by bounded function, that is, if $f\in L^\infty$ and $g\in L^{p,q}$  then $fg\in L^{p,q}$ and we have
\begin{equation}\label{eq:vi1}
\|fg\|_{L^{p,q}}\leq \|f\|_\infty \|g\|_{L^{p,q}}\, .
\end{equation}
Let $1\leq p,q<\infty$ and consider $0<\alpha<2$. The Lorentz potential space, $I_\alpha(L^{p,q}(\mathbb C))$, is the set of functions $f$ such that $f=I_\alpha*g$, where $g\in L^{p,q}(\mathbb C)$ and $I_\alpha (x) = c_{\alpha}|x|^{\alpha -2}$ is the Riesz potential of order $\alpha$. The norm in this space is given by
$$
\|f\|_{I_\alpha(L^{p,q}(\mathbb C))}=\|g\|_{L^{p,q}}.
$$
Note that when $\alpha =1$, one has $ \|f\|_{I_1(L^{p,q}(\mathbb C))}\approx \|\nabla f\|_{L^{p,q}}$.  

It is well known \cite{St2} that functions $f$ of $I_1(L^{2,1}(\mathbb C))$ are continuous and there exists a constant $C$ such that
\begin{equation}\label{eq:vi2}
\|f\|_\infty\leq C\|f\|_{I_1(L^{2,1}(\mathbb C))}.
\end{equation}
In general $I_\alpha(L^{\frac{2}{\alpha},1}(\mathbb C))$ are embedded in $\mathcal C_0$, the space of continuous functions vanishing at the infinity (see \cite{Ba}). Again, a remarkable property of these spaces $I_\alpha(L^{\frac{2}{\alpha},1}(\mathbb C))$ is that they are multiplication algebras, that is,
\begin{equation}\label{LorenAlg}
\|   fg\|_{I_\alpha(L^{\frac{2}{\alpha},1})} \le C  \|   f\|_{I_\alpha(L^{\frac{2}{\alpha},1})} \|   g\|_{I_\alpha(L^{\frac{2}{\alpha},1})} .
\end{equation}

Finally, note that  Calder\'on-Zygmund convolution operators are bounded on $L^{p,q}(\mathbb R^n)$ and so also on Lorentz potential space, $I_\alpha(L^{p,q}(\mathbb C))$, with constant depending on  \eqref{CZ}.

\section{Invertibility of the Beltrami operator}\label{sec3}

As we mentioned in the Introduction, to prove Proposition \ref{Prop1} (and then Theorem \ref{T1}) and Proposition \ref{Prop2} we only have to consider the invertibility of the Beltrami operator
$I-\mu\,B$ on $A^{s}_{p,q}(\C)$ and on $I_{1}(L^{2,1})  (\C)$. Following the idea of Iwaniec \cite[p. 42--43]{I1} we define
\[
P_{m}=I+\mu B+\cdots+(\mu B)^{m}\, ,
\]
so that we have
\begin{equation*}
(I-\mu B)P_{n-1}=P_{n-1}(I-\mu B)=I - (\mu B)^{n}    =I-\mu^{n}B^{n}+K ,\label{eq:fredh}
\end{equation*}
where $K= \mu^{n}B^{n} - (\mu B)^{n} $
can be easily seen to be a finite sum of operators that contain as a factor the commutator $[\mu , B] =
\mu B -B\mu$. In Lemma \ref{Le2} (and in Lemma \ref{Le2bis})  we will prove that $[\mu, B ]$  is compact on $A^{s}_{p,q}(\C)$ (and on $I_{1}(L^{2,1})  (\C)$) , so that $K$ is also compact.
In Lemma \ref{Le1} we will check that the operator norm of $\mu^{n}B^{n}$ on $A^{s}_{p,q}(\C)$ (and on $I_{1}(L^{2,1})  (\C)$) is small if $n$ is large.
 Therefore, $I - \mu B$  is a Fredholm operator on $A^{s}_{p,q}(\C)$ (and on $I_{1}(L^{2,1})  (\C)$).
 Clearly $I - t \mu B$, $0\le t\le 1$, is a continuous path from the identity to $I - \mu B$ . By the index theory of Fredholm 
operators on Banach spaces (e.g. \cite{Sch}), the index is a continuous function of the operator. Hence 
$I-\mu B$ has index 0. On the other hand, $I-\mu B$ is injective  on $A^{s}_{p,q}(\C)$ (and on $I_{1}(L^{2,1})  (\C)$) because by \cite[p. 43]{I1} it is injective on $L^p(\C)$ for all $1<p<\infty$. That concludes that $I-\mu B$ is invertible.

\begin{lemma}\label{Le1}
\begin{enumerate}[(a)]
\item The operator norm of $\mu^{n}B^{n}$ on $A^{s}_{p,q}(\C)$ is small if $n$ is large.
\item The operator norm of $\mu^{n}B^{n}$ on $I_{1}(L^{2,1}  (\C))$ is small if $n$ is large.
\end{enumerate}
\end{lemma}

\begin{proof}

Let $b_n=\dfrac{(-1)^n n}{\pi } \dfrac{\bar z^{n-1}}{z^{n+1}}$ the kernel of iterated Beurling transform $B^n$. Then, the Calder\'{o}n-Zygmund constant of  $B^n$ is
$$
\|b_{n}(z)|z|^{2}\|_{\infty}+\|\nabla
b_{n}(z)|z|^{3}\|_{\infty}\leq Cn^{2}.
$$

(a) It is an easy consequence of well-known results. Since $\| g^m \|_{A^{s}_{p,q}} \le C\|g \|^{m-1}_{\infty} \| g \|_{A^{s}_{p,q}}$ (see \cite[Teorem 5.3.2/4]{RS}), using  \eqref{algebra1} and \eqref{CZdes}, we have
\begin{eqnarray*}
\|\mu^{n}B^{n}(f)\|_{A^{s}_{p,q}} & \leq & C\:\|\mu^{n}\|_{A^{s}_{p,q}}\|B^{n}(f)\|_{A^{s}_{p,q}}  \\
& \leq & C\:\|\mu^{n}\|_{A^{s}_{p,q}}   n^2   \|f\|_{A^{s}_{p,q}}  \\
& \leq & C\: n^2 \|\mu \|^{n-1}_{\infty} \|\mu \|_{A^{s}_{p,q}}      \|f\|_{A^{s}_{p,q}} 
\end{eqnarray*}
and the norm becomes small if $n$ is big enough because $\| \mu \|_{\infty}= k<1$.
\\

(b) Using $ \|f\|_{I_1(L^{2,1})}\approx \|\nabla f\|_{L^{2,1}}$,  \eqref{LorenAlg}, \eqref{eq:vi1} and the boundedness of Calder\'on-Zygmund convolution operators, we have
\begin{eqnarray*}
\|\mu^{n}B^{n}(f)\|_{I_{1}(L^{2,1})} & \leq&  C\:\|\mu^{n}\|_{I_{1}(L^{2,1})}\|B^{n}(f)\|_{I_{1}(L^{2,1})}  \\
& \leq&  C\:\|\mu^{n}\|_{I_{1}(L^{2,1})}   n^2   \|f\|_{I_{1}(L^{2,1})} \\
& \leq&  C\: n^3 \|\mu \|^{n-1}_{\infty} \|\mu \|_{I_{1}(L^{2,1})}      \|f\|_{I_{1}(L^{2,1})}  
\end{eqnarray*}
and the norm becomes small if $n$ is big enough because $\| \mu \|_{\infty}= k<1$.

 \end{proof}

\begin{lemma}\label{Le2}
The commutator  $[\mu, B ]$  is compact on $A^{s}_{p,q}(\C)$.
\end{lemma}

\begin{proof}
First we have
\begin{eqnarray*}
\|[\mu,B]f\|_{A^{s}_{p,q}} & = & \|\mu Bf-B(\mu f)\|_{A^{s}_{p,q}}\\
  & \leq & \|\mu\|_{A^{s}_{p,q}}\|Bf\|_{A^{s}_{p,q}}+C \|\mu f\|_{A^{s}_{p,q}}\\
 & \leq  & C \|\mu\|_{A^{s}_{p,q}}\|f\|_{A^{s}_{p,q}}
 \end{eqnarray*}
and so the commutator is bounded in $A^{s}_{p,q}$.

Using that the limit of compact operators is a compact operator, we can assume that  $\mu \in\mathcal C_{c}^{\infty}(\mathbb{C})$, with its support contained in the disk $D(0, R)$. Now we use a trick from \cite[p. 145]{AIM}. Consider an arbitrary   function $g= \mathsf C \, f$ with $f\in A^{s}_{p,q} $, where $\mathsf C\, f$ denotes the Cauchy transform of $f$ 
(see \eqref{cauchy}). As $ \partial g= B(f)$, $\bar \partial g = f$ and $B( \bar\partial (\mu g))= \partial (\mu g)$,
\begin{equation*}
\begin{split}
\mu B(f) -B(\mu f) &  =  \mu \partial g - B(\mu \bar\partial g) = \mu \partial g - B( \bar\partial (\mu g)) +B(\bar\partial\mu  \, g)\\
& =   \mu \partial g -   \partial (\mu g)   + B(\bar\partial\mu  \, g) = B(\bar\partial\mu  \, g) - \partial\mu  \, g\\
& =  B(\bar\partial\mu  \, \mathsf C\, f) - \partial\mu  \, \mathsf C\, f
\end{split}
\end{equation*}
From this representation one can see that $[\mu, B ]$  is compact.
Given $\varphi\in \mathcal C_{c}^{\infty}(  D(0, R))$ the operator $\varphi\, \mathsf C\, f$ is a compact
operator on $A^{s}_{p,q}(\C)$, because by the lifting property (see \cite[2.1.4]{RS}) $\varphi\, \mathsf C\, f\in A^{s+1}_{p,q}(\C)$,  obviously $\varphi\, \mathsf C\, f(z)=0$ if $|z|\ge R$ and the inclusion of
$ A^{s+1}_{p,q}( D(0, R))$ into $ A^{s}_{p,q}(D(0, R))$ is compact (e.g. \cite[2.4.4]{RS}).

\end{proof}

\begin{lemma}\label{Le2bis}
The commutator  $[\mu, B ]$  is compact on $I_{1}(L^{2,1} (\C)$.
\end{lemma}
\begin{proof}
As above we have
\begin{equation}\label{acot}
\begin{split}
\|[\mu,B]f\|_{I_{1}(L^{2,1})} & =  \|\mu Bf-B(\mu f)\|_{I_{1}(L^{2,1})}\\
  & \leq  \|\mu\|_{I_{1}(L^{2,1})}\|Bf\|_{I_{1}(L^{2,1})}+C \|\mu f\|_{I_{1}(L^{2,1})}\\
 & \leq   C \|\mu\|_{I_{1}(L^{2,1})}\|f\|_{I_{1}(L^{2,1})}
 \end{split}
 \end{equation}
and so the commutator is bounded in $I_{1}(L^{2,1})$.
So, by density,  we only need to prove the compactness of the commutator when $\mu \in \mathcal C^{\infty}_c.$

On the other hand, 
\begin{eqnarray*}
\| [\mu,B]f\|_{I_{1}(L^{2,1})} & = & \sum_{j=1}^{2}\|\partial_{j}(\mu B(f)-B(\mu f))\|_{L^{2,1}}\\
& = &  \sum_{j=1}^{2}\|  [\partial_{j}\mu,B]f +  [\mu,B](\partial_{j}f)\|_{L^{2,1}}  \\
 & \leq & \sum_{j=1}^{2}\|[\partial_{j}\mu,B]f\|_{L^{2,1}}+\|[\mu,B](\partial_{j}f)\|_{L^{2,1}}.
 \end{eqnarray*}

Since the commutator is a compact operator in $L^p$ when $\mu$ is smooth  \cite{Uchi}  and using real interpolation of compact operators   \cite{CoP} we have that 
 $[\mu,B]\colon L^{2,1}(\mathbb{C})\to L^{2,1}(\mathbb{C})$ is compact.

Therefore we only have to prove that $[a,B]\colon I_{1}(L^{2,1}(\mathbb{C}))\to L^{2,1}(\mathbb{C})$
is a compact operator when $a\in \mathcal C_{c}^{\infty}(B(0,R))$ for some $R>0.$ Given $\eta >0$ we consider a regularization of the Beurling transform

$$
B^{\eta}f(z)= - \frac{1}{\pi}\,   \text{p.v.}  \int f(z-w)K_{\eta}(w) \mathrm{d}w  \,,
$$
where $K_{\eta}(z)=\dfrac{\varphi_{\eta}(z)}{z^2}$ and $0\le \varphi_{\eta}(z) \le 1$ is a radial 
$\mathcal C^{\infty}$ function satisfying 
$\varphi_{\eta}(|z|)=0$ if $|z|<\frac{\eta}{2}$ and  $\varphi_{\eta}(|z|)=1$ if $|z|> \eta.$
It is easy to check that $B^{\eta}$ is a convolution Calder\'{o}n-Zygmund operator with constants depending on $\eta.$ 

In the rest of this proof we will use the estimate  \eqref{eq:vi2} without any mention.
For any $f\in {I_{1}(L^{2,1})},$ the function $[a,B-B^{\eta}](f)$ has compact support. On the other hand, 

\begin{eqnarray*}
|[a,B-B^{\eta}](f)(z)| =\left|\frac{-1}{\pi}\int(a(z)-a(y))\left(\frac{1}{(z-y)^2}-\frac{\varphi_{\eta}(z-y)}{(z-y)^2}\right)f(y) \mathrm{d}y\right|\\[2mm]
   \le C\|f\|_{\infty}\|\nabla a\|_{\infty} \int_{|z-y|<\eta}\frac{1}{|z-y|}\,\mathrm{d}y\le C\eta\|f\|_{I_1( L^{2,1})}\|\nabla a\|_{\infty}.
\end{eqnarray*}

Consequently the operator $[a,B^{\eta}]$
tends to $[a,B]$ when  $\eta\to 0.$ To prove that  $[a,B^{\eta}]\colon I_{1}(L^{2,1})\to L^{2,1}$ 
is compact we will use Fr\'{e}chet-Kolgomorov Theorem for Lorentz spaces (e.g. \cite[p. 111]{Br1} for $L^p$ spaces).

By \eqref{acot}, the image by $[a,B^{\eta}]$ of the unit ball of $I_{1}(L^{2,1}(\mathbb C))$ is uniformly bounded in $L^{2,1}(\mathbb C).$
To get the equicontinuity, take $f\in I_{1}(L^{2,1})$ and $|z-w|<\frac{\eta}{8}.$ Then,

\begin{eqnarray*}
[a,B^{\eta}]f(z)-[a,B^{\eta}]f(w) & = & \frac{-1}{\pi}\left((a(z)-a(w)\right)\int_{\mathbb{C}}\frac{\varphi_{\eta}(z-\xi)}{(z-\xi)^2}f(\xi)\mathrm{d}\xi\\
 & + & \frac{-1}{\pi} \int_{\mathbb{C}}\left(\frac{\varphi_{\eta}(z-\xi)}{(z-\xi)^2}-\frac{\varphi_{\eta}(w-\xi)}{(w-\xi)^2}\right)\left(a(w)-a(\xi)\right)f(\xi)\mathrm{d}\xi\\
 & = & \theta_{1}(z,w)+\theta_{2}(z,w).
\end{eqnarray*} 
Since $B^{\eta}$ is a convolution Calder\'{o}n-Zygmund operator 
$$
|\theta_{1}(z,w)|  = \frac{1}{\pi}|\left(a(z)-a(w)\right)B^{\eta}f(z)| \le  C_{\eta}|z-w|\|\nabla a\|_{\infty}\|f\|_{I_{1}(L^{2,1})}
$$
and 
 \begin{eqnarray*}
|\theta_{2}(z,w)| & = & \frac{1}{\pi} \left|\int_{\mathbb{C}\setminus B(z,\frac{\eta}{4})}\left(\frac{\varphi_{\eta}(z-\xi)}{(z-\xi)^2}-
\frac{\varphi_{\eta}(w-\xi)}{(w-\xi)^2}\right)
\left(a(w)-a(\xi)\right)f(\xi)\mathrm{d}\xi\right|\\
 &  \leq &C |z-w|\|f\|_{\infty}\|a\|_{\infty}\left\{  \int_{|z-\xi|>\frac{\eta}{8}} \frac{1}{|z-\xi|^3} \mathrm{d}\xi + \int_{2\eta >|z-\xi|>\frac{\eta}{8}} \frac{\| \nabla \varphi_{\eta}  \|_{\infty}}{|z-\xi|^2} \mathrm{d}\xi \right\}  \\
&\leq & \frac{C}{\eta} |z-w| \|f\|_{I_1(L^{2,1})}.
\end{eqnarray*}

Therefore
\begin{equation}
|  [a,B^{\eta}]f(z)-[a,B^{\eta}]f(w)  | \le C |z-w| \|f\|_{I_1(L^{2,1})},\label{equi1}
\end{equation}
where the constant $C$ depends on $a$ and $\eta$.

 On the other hand, if $|z|>M>2R$ 
\begin{eqnarray*}
|[a,B^{\eta}]f(z)| & = & \left|\int_{\mathbb{C}}(a(z)-a(w))\frac{\varphi_{\eta}(z-w)}{(z-w)^2}f(w)\mathrm{d}w\right|\\
  & \leq & \|f\|_{\infty}\|a\|_{\infty}\int_{|w|<R}\frac{1}{|z-w|^2}\mathrm{d}w\\
  & \leq & C\|f\||_{I_1(L^{2,1})}\|a\|_{\infty}\frac{1}{|z|^{2}},
\end{eqnarray*} 
and then 
\begin{equation}
\|[a,B^{\eta}](f)\chi_{\mathbb C\setminus B(0,M)}\|_{L^{2,1}}\le C\|f\|_{I_1{L^{2,1}}}\|a\|_{\infty}\|\frac{1}{|z|^2}
\chi_{\mathbb C\setminus B(0,M)}\|_{L^{2,1}} ,\label{equi2}
\end{equation}
which tends to $0$ as $M\rightarrow 0$. Combining \eqref{equi1}  and \eqref{equi2}, by  Fr\'{e}chet-Kolgomorov Theorem for Lorentz spaces, one gets that  $[a,B^{\eta}]$ is a compact operator from $I_{1}(L^{2,1}(\mathbb{C}))$
to $L^{2,1}(\mathbb{C})$ as we desired.

\end{proof}

\section{ Calder\'{o}n-Zygmund operators on domains}\label{CZdo}

In this section we will prove  Theorem \ref{T3}. Let $X(\Omega)$ denote any of  function spaces in the statement of Theorem \ref{T3} and let $f\in X(\Omega)$. It is clear from the Calder\'{o}n-Zygmund theory that $T_{\Omega}f\in L^p(\Omega)$. So, in order to study the behaviour of $T_{\Omega}$ on $X(\Omega)$,  we must deal with  $T_{\Omega}f(x) -T_{\Omega}f(y)$ because we have a characterization of $X(\Omega)$ using first differences.  Following \cite{Me} we consider the next decomposition.

\begin{lemma} \label{lema:diferenciasT}

Let $\psi \in \mathcal{C}^{\infty}_{c}$ such that $\psi(u)=1$ on $|u|\leq2$
and $\psi(u)=0$ if $|u|\geq 4$. Define $\eta(u)=1-\psi(u)$. Then:  
\begin{equation*}
T_{\Omega}f(y)-T_{\Omega}f(x):=\sum_{i=1}^{4}g_{i}(x,y)+f(x)(T\chi_{\Omega}(y)-T\chi_{\Omega}(x)),\label{eq:difer}\end{equation*}
 where \begin{eqnarray*}
g_{1}(x,y) & = & \int_{\Omega}(K(y-u)-K(x-u))(f(u)-f(x))\ \eta\left(\frac{u-x}{|y-x|}\right)\mathrm{d}u,\\
g_{2}(x,y) & = & -\int_{\Omega}K(x-u)(f(u)-f(x))\ \psi\left(\frac{u-x}{|y-x|}\right)\mathrm{d}u,\\
g_{3}(x,y) & = & \int_{\Omega}K(y-u)(f(u)-f(y))\ \psi\left(\frac{u-x}{|y-x|}\right)\mathrm{d}u,\\
g_{4}(x,y) & = & (f(y)-f(x))\int_{\Omega}K(y-u)\ \psi\left(\frac{u-x}{|y-x|}\right)\mathrm{d}u.\end{eqnarray*}
 \end{lemma}
 
\begin{proof}
Note that if  $\tilde\psi(w) +\tilde\eta(w)=1$ we can write
\begin{eqnarray*}
T_{\Omega}f(x) & = & f(x)T_{\Omega}\tilde\psi(x)+\int_{\Omega}K(x- w)(f(w)-f(x))\ \tilde\psi(w)\mathrm{d}w\\
 &  & +\int_{\Omega} K(x- w)f(w)\tilde\eta(w)\mathrm{d}w,
 \end{eqnarray*}
 and then 
 \begin{eqnarray*}
 T_{\Omega}f(y)-T_{\Omega}f(x)& = & \int_{\Omega}(K(y-u)-K(x-u))(f(u)-f(x))\ \tilde\eta(u)\mathrm{d}u\\
 & - & \int_{\Omega}K(x-u)(f(u)-f(x))\  \tilde\psi(u)\mathrm{d}u\\
 & + & \int_{\Omega}K(y-u)(f(u)-f(y))\  \tilde\psi(u)\mathrm{d}u\\
 & + & (f(y)-f(x))\int_{\Omega}K(y-u)\  \tilde\psi(u)\mathrm{d}u\\
 & + & f(x)(T\chi_{\Omega}(y)-T\chi_{\Omega}(x)).
  \end{eqnarray*}
  Given $x\ne y$, take $\tilde\psi(u) = \psi\left(\frac{u-x}{|y-x|}\right)$ and $\tilde\eta(u) = \eta\left(\frac{u-x}{|y-x|}\right)$ and that is what we wished to prove.
 \end{proof}
 
 Let $B=B(x_{0},r)$ be the ball in $\mathbb R^n$ of center $x_{0}$ and radius $r$ and  $\varphi_{B}$ denotes a smooth function supported in $B$ such that $\| \varphi_{B} \|_{\infty}\le 1$ and $\| \nabla \varphi_{B} \|_{\infty}\le r^{-1}$. To deal with the term $g_{4}$ we will use the next lemma, which is an application of the Main Lemma of \cite{MOV}.
 
 \begin{lemma} \label{lema:acotacion}
 Let $\Omega$ be a bounded domain of $\mathbb R^n$ with boundary of class  $\mathcal{C}^{1,\beta}$, $\beta >0$, and let $T$ be an even smooth homogeneous Calder\'{o}n-Zygmund operator.
 Then, there exists a constant $C=C(\Omega)$ such that  $\|T_{\Omega}\varphi_{B}\|_{\infty}\leq C$. 
\end{lemma}

 \begin{proof} Since the $\mathcal{C}^{0,\beta}$ norm of $\varphi_{B}$ is bounded by $1+ r^{-\beta}$, by the Main Lemma of \cite{MOV} we have 
 $$
 \begin{array}{ll}
 | T_{\Omega}\varphi_{B}(x) | \le C(1+ r^{-\beta})\, ,  & \text{for all $x\in\mathbb C$ and} \\[2mm]
 | T_{\Omega}\varphi_{B}(x) -T_{\Omega}\varphi_{B}(y)   | \le C r^{-\beta} |x-y|^{\beta}, & \forall x,y \in \Omega.
  \end{array}
 $$
 
 Associated to the domain  $\Omega$ there is a $r_{0}>0$  satisfying \eqref{eq6}. Then, if $3r\ge r_{0}$  one has $ | T_{\Omega}\varphi_{B}(x) | \le C(1+ (\frac{3}{r_{0}})^{\beta})$ for all $x\in \mathbb C$. If $3r <r_{0}$ we write
\begin{eqnarray*}
T_{\Omega}\varphi_{B}(x) & = & \int_{\Omega}K(x-y)\varphi_{B}(y)\mathrm{d}y = \int_{\Omega\cap 3B}K(x-y)\varphi_{B}(y)\mathrm{d}y      \\
 & = & \int_{\Omega\cap 3B}K(x-y)(\varphi_{B}(y)-\varphi_{B}(x))\mathrm{d}y+\varphi_{B}(x)\int_{\Omega\cap 3B}K(x-y)\mathrm{d}y\\[1mm]
 & = & p(x)+ q(x)\end{eqnarray*}
 For $p(x)$ we have
  \[
|p(x)|\le C\int_{\Omega\cap 3B}\frac{|\varphi_{B}(x)-\varphi_{B}(y)|}{|x-y|^{n}}\mathrm{d}y\le C\|\nabla\varphi_{B}\|_{\infty}\int_{3B}\frac{\mathrm{d}y}{|x-y|^{n-1}}\leq C.
\]
If $x\notin B$, $q(x)=0$, and for $x\in B$ one can prove
$$
\left| \int_{\Omega\cap 3B}K(x-y)\mathrm{d}y\right| \le C(\Omega),
$$
proceeding as in the proof of the Main Lemma of  \cite[p. 408-410]{MOV}. Observe that for $x\in B$ the function $T_{\Omega} (\chi_{3B})$ has the same behaviour that $T_{\Omega} (1)=T(\chi_{\Omega})$ 

\end{proof}

Let's continuous with the proof of Theorem \ref{T3}. In the case that $f\in B^{s}_{p,p}(\Omega)$, $0<s<1$, $n<sp<\infty$, we have to prove that
 \begin{equation}
\int_{\Omega}\int_{\Omega}\frac{|T_{\Omega}f(x)-T_{\Omega}f(y)|^{p}}{|x-y|^{n+s p}}\mathrm{d}x\mathrm{d}y<\infty.\label{eq:carcoroBesov}
\end{equation}
 By Lemma \ref{lema:diferenciasT},
  \[
T_{\Omega}f(y)-T_{\Omega}f(x)=\sum_{i=1}^{4}g_{i}(x,y)+f(x)(T\chi_{\Omega}(y)-T\chi_{\Omega}(x))
\]
and we will study each term separately. Since $f$ is bounded (because $n<sp<\infty$) and $T\chi_{\Omega}\in  B^{s}_{p,p}(\Omega)$
 \begin{equation*}
\int_{\Omega}\int_{\Omega} \frac{|f(x)(T\chi_{\Omega}(y)-T\chi_{\Omega}(x))|^{p}}{|x-y|^{n+s p}}\mathrm{d}x\mathrm{d}y
\le \| f\|_{\infty} \int_{\Omega}\int_{\Omega} \frac{|T\chi_{\Omega}(y)-T\chi_{\Omega}(x)|^{p}}{|x-y|^{n+s p}}\mathrm{d}x\mathrm{d}y<\infty.
\end{equation*}

Fix $t$ such that $s<t<1$. Using the properties of the kernel $K$ and the H\"older's inequality ($ \frac{1}{p} +\frac{1}{q} =1$), we have
 \begin{eqnarray*}
|g_{1}(x,y)| & \le & C \int_{\Omega\cap\{|u-x|>2|x-y|\}}|K(x-u)-K(y-u)||f(u)-f(x)|\mathrm{d}u\\
 & \le & C \int_{\Omega\cap\{|u-x|>2|x-y|\}}\frac{|x-y|}{|x-u|^{n+1}}|f(u)-f(x)|\mathrm{d}u\\
& = & C |x-y|\int_{\Omega\cap\{|u-x|>2|x-y|\}}\frac{|f(u)-f(x)|}{|x-u|^{\frac{n}{p}+t}}\frac{1}{|x-u|^{\frac{n}{q}-t+1}}\mathrm{d}u\\
 & \le & C|x-y|\left(\int_{\Omega\cap\{|u-x|>2|x-y|\}}\frac{|f(u)-f(x)|^{p}}{|x-u|^{n+t p}}\mathrm{d}u\right)^{\frac{1}{p}}\\
 &  &\qquad\qquad  
 \cdot\left(\int_{\{|u-x|>2|x-y|\}}\frac{\mathrm{d}u}{|x-u|^{n-t q+q}}\right)^{\frac{1}{q}}\\
 & \le & C|x-y|^{t}\left(\int_{\Omega\cap\{|u-x|>2|x-y|\}}\frac{|f(u)-f(x)|^{p}}{|x-u|^{n+t p}}\mathrm{d}u\right)^{\frac{1}{p}}.
 \end{eqnarray*}
Thus
 \[
\frac{|g_{1}(x,y)|^{p}}{|x-y|^{n+s p}}\le \frac{C}{|x-y|^{n+s p-t p}}\int_{\Omega\cap\{|u-x|>2|x-y|\}}\frac{|f(u)-f(x)|^{p}}{|x-u|^{n+t p}}\mathrm{d}u,\]
and then, by the Fubini's theorem, 
\begin{align*}
\int_{\Omega}\int_{\Omega}\frac{|g_{1}(x,y)|^{p}}{|x-y|^{n+s p}} &\mathrm{d}x\mathrm{d}y \le \\ 
\le & C\int_{\Omega}\int_{\Omega}\frac{1}{|x-y|^{n+s p-t p}}\int_{\Omega\cap\{|u-x|>2|x-y|\}}\frac{|f(u)-f(x)|^{p}}{|x-u|^{n+t p}}\mathrm{d}u\mathrm{d}x\mathrm{d}y\\
 = & C\int_{\Omega}\int_{\Omega}\int_{\Omega\cap\{|u-x|>2|x-y|\}}\frac{|f(u)-f(x)|^{p}}{|x-u|^{n+p-s p}}\frac{1}{|x-y|^{n+s p-t p}}\mathrm{d}y\mathrm{d}u\mathrm{d}x\\
 \le & C\int_{\Omega}\int_{\Omega}\frac{|f(u)-f(x)|^{p}}{|x-u|^{n+t p}}\frac{1}{|x-u|^{s p-t p}}\mathrm{d}u\mathrm{d}x\\
 =  & C\int_{\Omega}\int_{\Omega}\frac{|f(u)-f(x)|^{p}}{|x-u|^{n+s p}}\mathrm{d}u\mathrm{d}x<\infty.
 \end{align*}
 
 Since the terms $g_{2}$ and $g_{3}$ are symmetric, we only consider one of them. Take $t$ such that $0<t<s$. As before, using the properties of the kernel $K$ and the H\"older's inequality ($ \frac{1}{p} +\frac{1}{q} =1$),
 \begin{eqnarray*}
|g_{2}(x,y)| & \le & C \int_{\Omega\cap\{|x-u|<4|x-y|\}}\frac{|f(u)-f(x)|}{|x-u|^{n}}\mathrm{d}u\\
 & = & C \int_{\Omega\cap\{|x-u|<4|x-y|\}}\frac{|f(u)-f(x)|}{|x-u|^{\frac{n}{p}+t}}\frac{1}{|x-u|^{\frac{n}{q}-t}}\mathrm{d}u\\
 & \le & C \left(\int_{\Omega\cap\{|x-u|<4|x-y|\}}\frac{|f(u)-f(x)|^{p}}{|x-u|^{n+t p}}\mathrm{d}u\right)^{\frac{1}{p}} \\
 &  & \qquad\qquad  \cdot\left(\int_{\{|x-u|<4|x-y|\}}\frac{\mathrm{d}u}{|x-u|^{n-t q}}\right)^{\frac{1}{q}}\\
 & \le & C |x-y|^{t}\left(\int_{\Omega\cap\{|x-u|<4|x-y|\}}\frac{|f(u)-f(x)|^{p}}{|x-u|^{n+t p}}\mathrm{d}u\right)^{\frac{1}{p}}.
 \end{eqnarray*}
Then
\begin{eqnarray*}
\frac{|g_{2}(x,y)|^{p}}{|x-y|^{n+s p}} & \le & \frac{C}{|x-y|^{n+s p- t p}}\int_{\Omega\cap\{|x-u|<4|x-y|\}}\frac{|f(u)-f(x)|^{p}}{|x-u|^{n+ t p}}\mathrm{d}u\end{eqnarray*}
 and therefore
 \begin{align*}
\int_{\Omega}\int_{\Omega}\frac{|g_{2}(x,y)|^{p}}{|x-y|^{n+s p}}&\mathrm{d}x\mathrm{d}y\le \\ 
\le & C \int_{\Omega}\int_{\Omega}\int_{\Omega\cap\{|x-u|<4|x-y|\}}\frac{|f(u)-f(x)|^{p}}{|x-y|^{n+s p-t p}|x-u|^{n+t p}}\mathrm{d}y\mathrm{d}u\mathrm{d}x\\
 \le & C \int_{\Omega}\int_{\Omega}\frac{|f(u)-f(x)|^{p}}{|x-u|^{n+s p}}\mathrm{d}u
 \mathrm{d}x < \infty.
 \end{align*}
 
Finally, by Lemma  \ref{lema:acotacion}  we have 
$$
\left|  \int_{\Omega}K(y-u)\ \psi\left(\frac{u-x}{|y-x|}\right)\mathrm{d}u\right|  \leq C
$$ 
and consequently
  \begin{equation*}
\int_{\Omega}\int_{\Omega}\frac{|g_{4}(x,y)|^{p}}{|x-y|^{n+s p}}\mathrm{d}u\mathrm{d}x\le C \int_{\Omega}\int_{\Omega}\frac{|f(x)-f(y)|^{p}}{|x-y|^{n+s p}}\mathrm{d}x\mathrm{d}y< \infty .
\end{equation*}

Combining  all these inequalities we get \eqref{eq:carcoroBesov}.\\

Using the characterizations \eqref{sobolevs} and \eqref{sobolev1} one can see that the proofs for $f\in W^{s,p}(\Omega)$ or  $f\in W^{1,p}(\Omega)$ are very similar to that we just explained for $f\in B^s_{p,p}(\Omega)$.\\

\textbf{Remark:}  If $\Omega$ is a bounded domain of $\mathbb R^n$ with boundary of class  $\mathcal{C}^{1,\varepsilon}$, $\varepsilon >0$, and  $T$ is an even smooth homogeneous Calder\'{o}n-Zygmund operator
we have (see \cite[Main Lemma]{MOV})
$$
|   T(\chi_{\Omega})( x) -  T(\chi_{\Omega})(y)        |\le C|x-y|^{\varepsilon}, \qquad \forall x,y\in\Omega.
$$
Therefore $ T(\chi_{\Omega})$ belongs to $W^{s,p}(\Omega)$ and $B_{p,p}^{s}(\Omega)$ for any $s\in (0,\varepsilon)$.

\section{Proof of the Theorem \ref{T2}}
Consider the Beurling transform restricted on the domain $\Omega$ of class  $\mathcal{C}^{1,\varepsilon}$
$$
B_{\Omega}g(z)= -\frac{1}{\pi}\int_{\Omega} \frac{g(w)}{(z-w)^{2}}\mathrm{d}w.
$$
By \cite[Main Lemma]{MOV}, $| B (\chi_{\Omega})(z) - B (\chi_{\Omega})(w) |\le C |z -w|^{\varepsilon}$ for all $z,w\in \Omega$.
Now, applying Theorem \ref{T3} we  have that $B_{\Omega}$ is bounded on the spaces $B_{p,p}^{s}(\Omega)$ and $W^{s,p}(\Omega)$. Let's denote by  $X(\Omega)$ any of these two spaces.
We will show that the Beltrami operator $I-\mu B_{\Omega}$ is invertible on  $X(\Omega)$. Then, taking $h= (I-\mu B_{\Omega})^{-1}(\mu)$ we get the conclusions.

As in the proof of the Proposition \ref{Prop1}, we claim that $I-\mu B_{\Omega}$ is a Fredholm operator on $X(\Omega)$. Define $P_{m}=I+\mu B_{\Omega}+\cdots+(\mu B_{\Omega})^{m}$
so that
 \begin{equation*}
(I-\mu B_{\Omega})P_{m-1}=P_{m-1}(I-\mu B_{\Omega})=I-\mu^{m}(B_{\Omega})^{m}+R, \label{eq:estrella}
\end{equation*}
where $R=\mu^{m}(B_{\Omega})^{m}-(\mu B_{\Omega})^{m}$ can be easily seen to be a finite sum of operators that contain the commutator $[\mu,B_{\Omega}]$ as a factor. We will prove that  $[\mu,B_{\Omega}]\colon X(\Omega)\to X(\Omega)$ is a compact operator. On the other hand, for $z\in \Omega$
\begin{align*}
(I-\mu^{m}(B_{\Omega})^{m})f(z) & =(I-\mu^{m}(B^{m})_{\Omega})f (z)+\mu^{m}(z)((B^{m})_{\Omega} f(z)-(B_{\Omega})^{m}f(z))\\
&  =(I-\mu^{m}(B^{m})_{\Omega})f (z)+\mu^{m}(z)K_{m} f(z),
\end{align*}
where $B^{m}$ is the $m$-iterated Beurling transform and $K_{m}f := (B^{m})_{\Omega} f -(B_{\Omega})^{m}f $. As in the proof of Lemma \ref{Le1}, if $F\in X(\Omega)$ we get
 \begin{equation}
\|\mu^{m} F\|_{X(\Omega)}\leq C\: m \|\mu\|_{\infty}^{m-1}  \|\mu\|_{X(\Omega)}
\| F\|_{X(\Omega)}.
\label{d11}
\end{equation}
Remind that the kernel of  $B^m$ is $\dfrac{(-1)^m m \bar z^{m-1} }{\pi  z^{m+1} }$ and then, by Theorem \ref{T3}, if $f\in X(\Omega)$ we have
\begin{equation}
\| (B^{m})_{\Omega} f\|_{X(\Omega)} \leq C\: m^{2}\| f \|_{X(\Omega)}.
\label{d12}
\end{equation}
Consequently, combining \eqref{d11} and \eqref{d12},
 \[
\|\mu^{m}(B^{m})_{\Omega} f\|_{X(\Omega)}\leq C\: m^{3} \|\mu\|_{\infty}^{m-1}  \|\mu\|_{X(\Omega)} \| f\|_{X(\Omega)},
\]
which implies that $I-\mu^{m}(B^{m})_{\Omega}$ is invertible if $m$ is large.  Assume for a moment that the operators $K_{m}$ are compacts on $X(\Omega)$. Thus,   $I-\mu B_{\Omega}$ is a Fredholm operator and in addition has index zero. Since  $X(\Omega) \subset L^p(\Omega)$ we also have that $I-\mu B_{\Omega}$ is injective (see \cite{I1}) and therefore invertible on $X(\Omega)$. 

The compactness of the operators $[\mu,B_{\Omega}]$ and $K_{m}$ on $X(\Omega)$ follows arguments parallels. Since $X(\Omega)$ is an algebra and the Beurling transform $B_{\Omega}$ is bounded on $X(\Omega)$ we have
\[
\|[\mu,B_{\Omega}]f\|_{X}  = \| \mu B_{\Omega}f -B_{\Omega}(\mu f) \|_{X}   \leq C\|\mu\|_{X}\|f\|_{X}.
\]
Moreover, because the domain $\Omega$ is Lipschitz, there exists a sequence of functions 
 $\mu_{j}\in C^{\infty}(\overline{\Omega})$ such that $\mu_{j}$ converges to $\mu$ in
 $X(\Omega)$. So, we have reduced to prove the compactness when $\mu\in C^{\infty}(\overline{\Omega})$. In this case, the kernel of the commutator
 \begin{equation*}
[ \mu,B_{\Omega}]f (z) = -\frac{1}{\pi}\int_{\Omega}\frac{\mu(z)-\mu(w)}{(z-w)^{2}}f(w) \mathrm{d}w
:= \int_{\Omega} k(z,w) f(w) \mathrm{d}w
\end{equation*}
clearly satisfies 
\begin{eqnarray*}
|k(z,w)| & \le & \frac{C}{|z-w|}\quad\text{for all } z,w\in\Omega,\label{eq:Pnuc1}\\
|k(z',w)-k(z,w)| & \le & C \frac{|z-z'|}{|z-w|^{2}}\quad\text{if \ensuremath{|z-w|>2|z-z'|}.}\label{eq:Pnuc2}\end{eqnarray*}
Then, a simple computation gives (see \cite[p. 419]{MOV}), for $ z_1\,,\,z_2 \in \Omega$,
\begin{equation}\label{eq18}
|[\mu,B_{\Omega}]f(z_1)- [\mu,B_{\Omega}]f (z_2)| \le C\,|z_{1}-z_{2}| \,(1+\log
\frac{d}{|z_{1}-z_{2}|})\,\|f\|_\infty\,, 
\end{equation}
where $d$ denotes the diameter of $\Omega$. From \eqref{eq18} one immediately gets that 
$[ \mu,B_{\Omega}]f$ belongs to $B_{p,p}^{\beta}(\Omega)$ and to $W^{\beta, p}(\Omega)$ for any
$\beta <1$. The compact embedding $W^{\beta, p}(\Omega) \hookrightarrow W^{s, p}(\Omega)$, $s<\beta$, (and $B_{p,p}^{\beta}(\Omega) \hookrightarrow  B_{p,p}^{s}(\Omega)$) gives the compactness for the commutator (e.g. \cite[Proposition 7]{Tri2}).\\

We have  $K_{m}f = (B^{m})_{\Omega} f -(B_{\Omega})^{m}f $.
To prove that $K_{m}$ is compact on $X(\Omega)$ we will proceed by induction. For $m\geq2$,  
\begin{eqnarray*}
(B_{\Omega})^{m}f & = & B_{\Omega}((B_{\Omega})^{m-1}f)  =  B([(B_{\Omega})^{m-1}f] \chi_{\Omega})\\
& = & B([ B^{m-1}(f \chi_{\Omega}) -  K_{m-1}f] \chi_{\Omega})\\
& = & B( B^{m-1}(f \chi_{\Omega}) - (B^{m-1}(f \chi_{\Omega}))\chi_{\Omega^{c}} -  (K_{m-1}f )\chi_{\Omega})\\
 & = & B^{m}(f \chi_{\Omega}) - B(\chi_{\Omega^{c}} B^{m-1}(f \chi_{\Omega})  ) - B_{\Omega}(K_{m-1}f)
 \end{eqnarray*}
 It is then enough to prove that, for  $m\geq1$, the operator 
 \[
 Q_{m} f := B( (B^{m} (f\chi_{\Omega})) \chi_{\Omega^{c}}))
\]
is compact in $X(\Omega)$. For $z\in \Omega$, we write
\begin{eqnarray*}
Q_{m}f (z) & = &  B(  (B^{m} (f \chi_{\Omega})) \chi_{\Omega^{c}} ) (z) \\ [2mm]
& = & -\frac{1}{\pi}\int_{\Omega^{c}}\frac{B^{m}(f \chi_{\Omega})(w)}{(z-w)^{2}}\mathrm{d}w\\
& = & - \frac{1}{\pi} \int_{\Omega^c} \frac{1}{(z-w)^2} \frac{(-1)^m m}{\pi}\,\int_\Omega
\frac{(\overline{w-\xi})^{m-1}}{(w-\xi)^{m+1}}\,f(\xi)\, \mathrm{d} \xi
\,\mathrm{d} w \\
 & = & \int_{\Omega}K_{m}(z,\xi)f(\xi)\mathrm{d}\xi,\end{eqnarray*}
where 
\[
K_{m}(z,\xi):= \frac{(-1)^{m+1}}{\pi^{2}} \int_{\Omega^{c}}\frac{1}{(z-w)^{2}}\frac{m\overline{(w-\xi})^{m-1}}{(w-\xi)^{m+1}}\mathrm{d}w .
\]
In \cite[p. 418--419]{MOV}, it is proved that if $\Omega$ is a bounded domain of class $\mathcal{C}^{1,\varepsilon}$ and $f\in L^{\infty}(\Omega)$ then 
\begin{eqnarray*}
| Q_{m}f (z)| & \le & C d^{\varepsilon} \| f\|_{\infty}\,  , \quad z\in \Omega \, , \\
| Q_{m}f(z_1)- Q_{m}f (z_2)|&  \le & C\,|z_{1}-z_{2}|^{\varepsilon} \,(1+\log
\frac{d}{|z_{1}-z_{2}|})\,\|f\|_{\infty}\, , \quad z_{1}, z_{2}\in \Omega \, ,
\end{eqnarray*}
where $d$ denotes the diameter of $\Omega$ and $C$ depends on $m$ and $\Omega$.

Consequently, if $f\in X(\Omega)$ then $Q_{m}f$ belongs to $B_{p,p}^{\beta}(\Omega)$ and to $W^{\beta, p}(\Omega)$ for any $\beta <\varepsilon$. Choose $\beta$ such that $s<\beta< \varepsilon$. Again, the compact embeddings 
$W^{\beta, p}(\Omega) \hookrightarrow W^{s, p}(\Omega)$ and $B_{p,p}^{\beta}(\Omega) \hookrightarrow  B_{p,p}^{s}(\Omega)$  give the compactness of $Q_{m}$.

\begin{gracies}
The authors were partially supported by grants 2009SGR420
(Generalitat de Catalunya) and  MTM2010-15657 (Ministerio de Ciencia
e Innovaci\'{o}n, Spain).

\end{gracies}

\begin{tabular}{l}
Victor Cruz\\
Instituto de F\'{\i}sica y Matem\'aticas\\
Universidad Tecnol\'ogica de la Mixteca\\
69000 Huajuapan, Oaxaca, M\'exico\\
{\it E-mail:} {\tt victorcruz@mixteco.utm.mx}\\ \\
Joan Mateu\\
Departament de Matem\`{a}tiques\\
Universitat Aut\`{o}noma de Barcelona\\
08193 Bellaterra, Barcelona, Catalonia\\
{\it E-mail:} {\tt mateu@mat.uab.cat}\\ \\
Joan Orobitg\\
Departament de Matem\`{a}tiques\\
Universitat Aut\`{o}noma de Barcelona\\
08193 Bellaterra, Barcelona, Catalonia\\
{\it E-mail:} {\tt orobitg@mat.uab.cat}\\ \\
\end{tabular}

\end{document}